\documentclass[a4paper,oneside,11pt]{article}%
\usepackage{makeidx}
\usepackage[english]{babel}
\usepackage{amsmath}
\usepackage{amsfonts}
\usepackage{amssymb}
\usepackage{stmaryrd}
\usepackage{graphicx}
\usepackage{mathrsfs}
\usepackage{cite}
\usepackage[colorlinks,linkcolor=red,anchorcolor=blue,citecolor=blue,urlcolor=blue]{hyperref}
\usepackage[symbol*,ragged]{footmisc}
\providecommand{\U}[1]{\protect\rule{.1in}{.1in}}
\providecommand{\U}[1]{\protect \rule{.1in}{.1in}}

\newtheorem{theorem}{Theorem}[section]

\newtheorem{lemma}[theorem]{Lemma}

\newenvironment{proof}[1][Proof]{\noindent \textbf{#1.} }{\  \rule{0.5em}{0.5em}}
\numberwithin{equation}{section}

\begin{document}

\title{On the Fourth Power Moment of Fourier \\ Coefficients of Cusp Form}

\author{Jinjiang Li\footnotemark[1]\,\,\,\, \, \& \,\, Panwang Wang\,\footnotemark[2]\,\,\,\, \, \& \,\, Min Zhang\footnotemark[3]    \vspace*{-4mm} \\
$\textrm{\small Department of Mathematics, China University of Mining and Technology}^{*\,\dag\,\ddag}$
                    \vspace*{-4mm} \\
     \small  Beijing 100083, P. R. China  }

\footnotetext[3]{Corresponding author. \\
    \quad\,\, \textit{ E-mail addresses}: \href{mailto:jinjiang.li.math@gmail.com}{jinjiang.li.math@gmail.com} (J. Li),
    \href{mailto:panwangw@gmail.com}{panwangw@gmail.com} (P. Wang), \\
    \qquad\qquad\qquad\qquad\qquad \href{mailto:min.zhang.math@gmail.com}{min.zhang.math@gmail.com} (M. Zhang).   }

\date{}
\maketitle


{\textbf{Abstract}}: Let~$a(n)$~be the Fourier coefficients of a holomorphic cusp form of weight~$\kappa=2n\geqslant12$~for the full
modular group and $A(x)=\sum\limits_{n\leqslant x}a(n)$. In this paper, we establish
an asymptotic formula of the fourth power moment of $A(x)$ and prove that
\begin{equation*}
  \int_1^TA^4(x)\mathrm{d}x=\frac{3}{64\kappa\pi^4}s_{4;2}(\tilde{a}) T^{2\kappa}+O\big(T^{2\kappa-\delta_4+\varepsilon}\big)
\end{equation*}
with $\delta_4=1/8$, which improves the previous result.



\section{Introduction and main result}

   Let~$a(n)$~be the Fourier coefficients of a holomorphic cusp form of weight~$\kappa=2n\geqslant12$~for the full modular group. In 1974,
   Deligne\cite{Deligne} proved the following profound result
\begin{equation}\label{a-d}
    a(n)\ll n^{(\kappa-1)/2}d(n),
\end{equation}
 where~$d(n)$~denotes the Dirichlet divisor function and the implied constant in~$\ll$~is absolute.
Suppose~$x\geqslant2$~and define
\begin{equation}
  A(x):=\sum_{n\leqslant x}a(n).
\end{equation}
It is well known that~$A(x)$~has no main term and~$A(x)\ll x^{\kappa/2-1/6+\varepsilon}$.
 In 1973, Joris\cite{Joris} proved that
\begin{equation*}
   A(x)=\Omega_{\pm}\big(x^{\kappa/2-1/4}\log\log\log x\big).
\end{equation*}
In 1990, Ivi\'{c}\cite{Ivic} showed that there exist two points~$t_1$~and~$t_2$~in the interval~$[T,T+CT^{1/2}]$~such that
\begin{equation*}
  A(t_1)>Bt_1^{\kappa/2-1/4},\qquad A(t_2)<-Bt_2^{\kappa/2-1/4},
\end{equation*}
where~$B>0,\,C>0$~are constants. It is conjectured that
\begin{equation*}
  A(x)\ll x^{(\kappa-1)/2+1/4+\varepsilon}
\end{equation*}
is true for every~$\varepsilon.$~The evidence in support of this conjecture has been given by Ivi\'{c}\cite{Ivic}, who proved the following
square mean value formula of~$A(x)$, i.e.
\begin{equation*}
  \int_{1}^{T}A^2(x)\mathrm{d}x=\mathcal{C}_2T^{\kappa+1/2}+B(T),
\end{equation*}
where
\begin{equation*}
  \mathcal{C}_2=\frac{1}{(4\kappa+2)\pi^2}\sum_{n=1}^{\infty}a^2(n)n^{-\kappa-1/2},
\end{equation*}
\begin{equation*}
  B(T)\ll T^{\kappa}\log^5T,\quad B(T)=\Omega\bigg(T^{\kappa-1/4}\frac{(\log\log\log T)^3}{\log T}\bigg).
\end{equation*}
In\cite{Ivic}, Ivi\'{c} also proved the upper bound of eighth power moment of~$A(x)$,~that is
\begin{equation*}
   \int_1^TA^8(x)\mathrm{d}x\ll T^{4\kappa-1+\varepsilon}.
\end{equation*}
Cai\cite{Cai} studied the third and fourth power moments of~$A(x)$. He proved that
\begin{equation}\label{3-power}
  \int_1^TA^3(x)\mathrm{d}x=\mathcal{C}_3 T^{(6\kappa+1)/4}+O(T^{(6\kappa+1)/4-\delta_3+\varepsilon}),
\end{equation}
\begin{equation}\label{4-power}
  \int_1^TA^4(x)\mathrm{d}x=\mathcal{C}_4 T^{2\kappa}+O(T^{2\kappa-\delta_4+\varepsilon}),
\end{equation}
where~$\delta_3=1/14,\,\delta_4=1/23$~and
\begin{equation*}
  \mathcal{C}_3:=\frac{3}{4(6\kappa+1)\pi^3}\sum_{\substack{n,m,k\in\mathbb{N} \\ \sqrt{n}+\sqrt{m}=\sqrt{k}}}
               (nmk)^{-\kappa/2-1/4} a(n)a(m)a(k),
\end{equation*}
\begin{equation*}
  \mathcal{C}_4:=\frac{3}{64\kappa\pi^4}\sum_{\substack{n,m,k,\ell\in\mathbb{N} \\ \sqrt{n}+\sqrt{m}=\sqrt{k}+\sqrt{\ell}}}
    (nmk\ell)^{-\kappa/2-1/4} a(n)a(m)a(k)a(\ell).
\end{equation*}

  In\cite{Zhai-1}, Zhai proved that (\ref{3-power}) holds for~$\delta_3=1/4.$~Following the approach of Tsang\cite{Tsang}, Zhai\cite{Zhai-1} proved that the equation (\ref{4-power}) holds for~$\delta_4=2/41.$~This approach used the method of exponential sums. In particular, if the exponent
pair conjecture is true, namely, if~$(\varepsilon,1/2+\varepsilon)$~ is an exponent pair, then the equation~(\ref{4-power})~holds for~$\delta_4=1/14$. Later, combining the method of~\cite{Ivic-Sargos} and a deep result of Robert and Sargos\cite{Robert-Sargos}, Zhai\cite{Zhai-3} proved that the
equation (\ref{4-power}) holds for~$\delta_4=3/28$.~By a unified approach, Zhai\cite{Zhai-2} proved that the asymptotic formula
\begin{equation*}
   \int_1^TA^k(x)\mathrm{d}x=\mathcal{C}_kT^{1+k(2\kappa-1)/4}+O\big(T^{1+k(2\kappa-1)/4-\delta_k+\varepsilon}\big)
\end{equation*}
holds for~$3\leqslant k\leqslant7,$~where~$\mathcal{C}_k$~and~$0<\delta_k<1$~are explicit constants.

 The aim of this paper is to improve the value of~$\delta_4=3/28$,~which is achieved by Zhai\cite{Zhai-3}. The main result is the following
\begin{theorem}\label{4-power-theorem}
   We have
   \begin{equation*}
    \int_1^TA^4(x)\mathrm{d}x=\frac{3}{64\kappa\pi^4}s_{4;2}(\tilde{a}) T^{2\kappa}+O\big(T^{2\kappa-\delta_4+\varepsilon}\big)
   \end{equation*}
   with~$\delta_4=1/8,$~where
   \begin{equation*}
      s_{4;2}(\tilde{a})=\sum_{\substack{n,m,k,\ell\in\mathbb{N}^*\\ \sqrt{n}+\sqrt{m}=\sqrt{k}+\sqrt{\ell}}}
          \frac{a(n)a(m)a(k)a(\ell)}{(nmk\ell)^{\kappa/2+1/4}}.
    \end{equation*}
\end{theorem}

\textbf{Notation.} Throughout this paper, $a(n)$ be the Fourier coefficients of a holomorphic cusp form of weight~$\kappa=2n\geqslant12$~for the full modular group; $d(n)$ denote the Dirichlet divisor function; $\tilde{a}(n):=a(n)n^{-\kappa/2+1/2}$;~$\|x\|$~denotes the distance from $x$ to the nearest integer, i.e., $\|x\|=\min\limits_{n\in\mathbb{Z}}|x-n|$. $[x]$~denotes the integer part of~$x$;~$n\sim N$~means~$N<n\leqslant2N$;~$n\asymp N$~means~$C_1N\leqslant n\leqslant C_2N$~with positive constants $C_1,\,C_2$
satisfying $C_1<C_2.$ $\varepsilon$ always denotes an arbitrary small positive constant which may not be the same at different occurances. We shall use the estimates $d(n)\ll n^{\varepsilon}$. Suppose $f:\mathbb{N}\to\mathbb{R}$ is any function satisfying $f(n)\ll n^\varepsilon,\,k\geqslant2$ is a fixed integer.  Define
\begin{equation}\label{s-k-f}
   s_{k;\ell}(f):=\sum_{\substack{n_1,\cdots,n_\ell,n_{\ell+1},\cdots,n_k\in\mathbb{N}^* \\
             \sqrt{n_1}+\cdots+\sqrt{n_l}=\sqrt{n_{\ell+1}}+\cdots+\sqrt{n_k}}}
             \frac{f(n_1)f(n_2)\cdots f(n_k)}{(n_1n_2\cdots n_k)^{3/4}},\qquad 1\leqslant\ell<k.
\end{equation}
We shall use $s_{k;\ell}(f)$  to denote both of the series (\ref{s-k-f}) and its value. Suppose $y>1$  is a large parameter,
and we define
\begin{equation*}
   s_{k;\ell}(f;y):=\sum_{\substack{n_1,\cdots,n_\ell,n_{\ell+1},\cdots,n_k\leqslant y \\
             \sqrt{n_1}+\cdots+\sqrt{n_l}=\sqrt{n_{\ell+1}}+\cdots+\sqrt{n_k}}}
             \frac{f(n_1)f(n_2)\cdots f(n_k)}{(n_1n_2\cdots n_k)^{3/4}},\qquad 1\leqslant\ell<k.
\end{equation*}

\section{Preliminary Lemmas }

\begin{lemma}\label{Tsang-lemma}
  If $g(x)$ and $h(x)$ are continuous real-valued functions of $x$ and $g(x)$ is monotonic, then
  \begin{equation*}
     \int_{a}^{b}g(x)h(x)\mathrm{d}x\ll \Big(\max_{a\leqslant x\leqslant b}|g(x)|\Big)
     \bigg(\max_{a\leqslant u<v\leqslant b}\bigg|\int_u^v h(x)\mathrm{d}x\bigg|\bigg).
  \end{equation*}
\end{lemma}
\begin{proof}
  See Tsang \cite{Tsang}, Lemma 1.
\end{proof}

\begin{lemma}\label{cos-estimate}
  Suppose $A,\,B\in \mathbb{R},\,A\not=0.$ . Then we have
  \begin{equation*}
     \int_T^{2T}t^\alpha\cos(A\sqrt{t}+B)\mathrm{d}t\ll T^{1/2+\alpha}|A|^{-1}.
  \end{equation*}
\end{lemma}
\begin{proof}
   It follows from Lemma \ref{Tsang-lemma} easily.
\end{proof}

\begin{lemma}\label{kong-lamma}
   If $n,m,k,\ell\in\mathbb{N}$ such that $\sqrt{n}+\sqrt{m}\pm\sqrt{k}-\sqrt{\ell}\not=0$, then there hold
   \begin{equation*}
      |\sqrt{n}+\sqrt{m}\pm\sqrt{k}-\sqrt{\ell}|\gg(nmk\ell)^{-1/2}\max(n,m,k,\ell)^{-3/2},
   \end{equation*}
    respectively.
\end{lemma}
\begin{proof}
   See Kong~\cite{Kong}, Lemma~3.2.1.
\end{proof}

\begin{lemma}\label{error-lamma}
  Let $f:\mathbb{N}\to\mathbb{R}$ be any function satisfying $f(n)\ll n^{\varepsilon}$. Then we have
  \begin{equation*}
     \big|s_{k;\ell}(f)-s_{k;\ell}(f;y)\big|\ll y^{-1/2+\varepsilon}, \quad 1\leqslant\ell< k,
  \end{equation*}
  where $k\geqslant2$ is a fixed integer.
\end{lemma}
\begin{proof}
   See Zhai~\cite{Zhai-2}, Lemma~3.1.
\end{proof}

\begin{lemma}\label{Zhai-lemma-5}
   Suppose $1\leqslant N\leqslant M,\,1\leqslant K\leqslant L,\,N\leqslant K,\,M\asymp L,\,0<\Delta\ll L^{1/2}$. Let $\mathscr{A}_1(N,M,K,L;\Delta)$
denote the number of solutions of the following inequality
\begin{equation*}
    0<|\sqrt{n}+\sqrt{m}-\sqrt{k}-\sqrt{\ell}|<\Delta
\end{equation*}
with $n\sim N,\,m\sim M,\,k\sim K,\,\ell\sim L$. Then we have
\begin{equation*}
   \mathscr{A}_1(N,M,K,L;\Delta)\ll \Delta L^{1/2}NMK+NKL^{1/2+\varepsilon}.
\end{equation*}
Especially, if $\Delta L^{1/2}\gg1$, then
\begin{equation*}
  \mathscr{A}_1(N,M,K,L;\Delta)\ll \Delta L^{1/2}NMK.
\end{equation*}
\end{lemma}
\begin{proof}
   See Zhai~\cite{Zhai-3}, Lemma~5.
\end{proof}

\begin{lemma}\label{Zhai-lemma-3}
   Suppose $N_j\geqslant2\,(j=1,2,3,4),\,\Delta>0$. Let $\mathscr{A}_\pm(N_1,N_2,N_3,N_4;\Delta)$ denote the number of
     solutions of the following inequality
\begin{equation*}
   0<|\sqrt{n_1}+\sqrt{n_2}\pm\sqrt{n_3}-\sqrt{n_4}|<\Delta
\end{equation*}
with $n_j\sim N_j\,(j=1,2,3,4),\,n_j\in\mathbb{N}^*$. Then we have
\begin{equation*}
  \mathscr{A}_\pm(N_1,N_2,N_3,N_4;\Delta)\ll\prod_{j=1}^4\big(\Delta^{1/4}N_j^{7/8}+N_j^{1/2}\big)N_j^{\varepsilon}.
\end{equation*}
\end{lemma}
\begin{proof}
   See Zhai~\cite{Zhai-3}, Lemma~3.
\end{proof}

\section{Proof of Theorem~\ref{4-power-theorem}}
In this section, we shall prove the theorem. We begin with the following truncated formula, which is proved by Jutila\cite{Jutila}, i.e.,
\begin{equation}\label{truncated-formula}
    A(x)=\frac{1}{\sqrt{2}\pi}\sum_{n\leqslant N}\frac{a(n)}{n^{\kappa/2+1/4}}x^{\kappa/2-1/4}\cos(4\pi\sqrt{nx}-\pi/4)+O(x^{\kappa/2+\varepsilon } N^{-1/2}),
\end{equation}
where $1\leqslant N\ll x.$

 Suppose $T\geqslant10.$ By a splitting argument, it is sufficient to prove the result in the interval $[T,2T]$. Take $y=T^{3/4}$. For any
 $T\leqslant x\leqslant 2T$, by the truncated formula (\ref{truncated-formula}), we get
\begin{equation}
   A(x)=\frac{1}{\sqrt{2}\pi} \mathcal{R}(x)+O(x^{\kappa/2+\varepsilon}y^{-1/2}),
\end{equation}
where
\begin{equation*}
   \mathcal{R}(x):=x^{\kappa/2-1/4}\sum_{n\leqslant y}\frac{a(n)}{n^{\kappa/2+1/4}}\cos(4\pi\sqrt{nx}-\pi/4).
\end{equation*}
We have
\begin{eqnarray}\label{A^4-asymp}
   \int_T^{2T}A^4(x)\mathrm{d}x
   & = & \frac{1}{4\pi^4}\int_T^{2T}\mathcal{R}^{4}(x)\mathrm{d}x+O\big(T^{2\kappa+1/4+\varepsilon}y^{-1/2}+T^{2\kappa+1+\varepsilon}y^{-2}\big)  \nonumber \\
   & = & \frac{1}{4\pi^4}\int_T^{2T}\mathcal{R}^{4}(x)\mathrm{d}x+O(T^{2\kappa-1/8+\varepsilon}).
\end{eqnarray}

Let
\begin{equation*}
  g=g(n,m,k,\ell):=\left\{
  \begin{array}{cc}
        \displaystyle\frac{a(n)a(m)a(k)a(\ell)}{(nmk\ell)^{\kappa/2+1/4}}, & \textrm{if\,\,\,\,} n,m,k,\ell\leqslant y, \\
        0, & \textrm{otherwise}.
  \end{array}
  \right.
\end{equation*}
According to the elementary formula
\begin{equation*}
\cos{a_1}\cos{a_2}\cdots\cos{a_h}=\frac{1}{2^{h-1}}\sum_{(i_1,i_2\cdots,i_{h-1})\in\{0,1\}^{h-1}}\cos\big(a_1+(-1)^{i_1}a_2+\cdots+(-1)^{i_{h-1}}a_h\big),
\end{equation*}
we can write
\begin{equation}\label{R^4-fenjie}
 \mathcal{R}^4(x)=S_1(x)+S_2(x)+S_3(x)+S_4(x),
\end{equation}
where
\begin{eqnarray*}
  S_1(x) & := & \frac{3}{8} \sum_{\substack{n,m,k,\ell\leqslant y\\ \sqrt{n}+\sqrt{m}=\sqrt{k}+\sqrt{\ell}}} gx^{2\kappa-1}, \\
  S_2(x) & := & \frac{3}{8} \sum_{\substack{n,m,k,\ell\leqslant y\\ \sqrt{n}+\sqrt{m}\not=\sqrt{k}+\sqrt{\ell}}} gx^{2\kappa-1}
                            \cos\big(4\pi(\sqrt{n}+\sqrt{m}-\sqrt{k}-\sqrt{\ell})\sqrt{x}\big), \\
\end{eqnarray*}
\begin{eqnarray*}
  S_3(x) & := & \frac{1}{2} \sum_{\substack{n,m,k,\ell\leqslant y\\ \sqrt{n}+\sqrt{m}+\sqrt{k}\not=\sqrt{\ell}}} gx^{2\kappa-1}
                            \cos\Big(4\pi(\sqrt{n}+\sqrt{m}+\sqrt{k}-\sqrt{\ell})\sqrt{x}-\frac{\pi}{2}\Big), \\
  S_4(x) & := & \frac{1}{8} \sum_{n,m,k,\ell\leqslant y}gx^{2\kappa-1}
                            \cos\big(4\pi(\sqrt{n}+\sqrt{m}+\sqrt{k}+\sqrt{\ell})\sqrt{x}-\pi\big).
\end{eqnarray*}
By (\ref{a-d}) and Lemma \ref{error-lamma}, we get
\begin{eqnarray}\label{S_1-estimate}
  \int_T^{2T}S_1(x)\mathrm{d}x
     & = & \frac{3}{8}s_{4;2}\big(a(n)n^{-\kappa/2+1/2};y\big)\int_T^{2T}x^{2\kappa-1}\mathrm{d}x   \nonumber \\
     & = & \frac{3}{8}s_{4;2}(\tilde{a};y)\int_T^{2T}x^{2\kappa-1}\mathrm{d}x  \nonumber \\
     & = & \frac{3}{8}s_{4;2}(\tilde{a})\int_T^{2T}x^{2\kappa-1}\mathrm{d}x+O(T^{2\kappa}y^{-1/2+\varepsilon})  \nonumber \\
      & = & \frac{3}{8}s_{4;2}(\tilde{a})\int_T^{2T}x^{2\kappa-1}\mathrm{d}x+O(T^{2\kappa-3/8+\varepsilon}).
\end{eqnarray}

 We now proceed to consider the contribution of $S_4(x)$. Applying Lemma \ref{cos-estimate} and (\ref{a-d}), we get
\begin{eqnarray}\label{S_4-estimate}
  \int_T^{2T} S_4(x)\mathrm{d}x
        & = &    \frac{1}{8}\sum_{n,m,k,\ell\leqslant y}g\int_{T}^{2T}x^{2\kappa-1}
                 \cos\big(4\pi(\sqrt{n}+\sqrt{m}+\sqrt{k}+\sqrt{\ell})\sqrt{x}-\pi\big)\mathrm{d}x   \nonumber \\
        & \ll &  \sum_{n,m,k,\ell\leqslant y}\frac{gT^{2\kappa-1/2}}{\sqrt{n}+\sqrt{m}+\sqrt{k}+\sqrt{\ell}}  \nonumber \\
        & = &    T^{2\kappa-1/2}\sum_{n,m,k,\ell\leqslant y}\frac{a(n)a(m)a(k)a(\ell)}{(nmk\ell)^{(\kappa-1)/2}(nmk\ell)^{3/4}}
                 \cdot\frac{1}{\sqrt{n}+\sqrt{m}+\sqrt{k}+\sqrt{\ell}}   \nonumber \\
        & \ll &  T^{2\kappa-1/2}\sum_{n,m,k,\ell\leqslant y}\frac{d(n)d(m)d(k)d(\ell)}{(nmk\ell)^{3/4}\ell^{1/2}}    \nonumber \\
        & \ll &  T^{2\kappa-1/2+\varepsilon}\sum_{n,m,k,\ell\leqslant y}\frac{1}{(nmk)^{3/4}\ell^{5/4}}     \nonumber \\
        & \ll &  T^{2\kappa-1/2+\varepsilon}y^{1/2} \ll T^{2\kappa-1/8+\varepsilon}.
\end{eqnarray}

  Now let us consider the contribution of $S_2(x)$. By the first derivative test. we have
\begin{eqnarray}\label{S_2-estimate}
  \int_T^{2T}S_2(x)\mathrm{d}x
      & \ll & \sum_{\substack{n,m,k,\ell\leqslant y\\ \sqrt{n}+\sqrt{m}\not=\sqrt{k}+\sqrt{\ell}}}
              g\min\bigg(T^{2\kappa},\frac{T^{2\kappa-1/2}}{|\sqrt{n}+\sqrt{m}-\sqrt{k}-\sqrt{\ell}|}\bigg)    \nonumber \\
      & \ll & x^{\varepsilon} \mathcal{G}(N,M,K,L),
\end{eqnarray}
where
\begin{equation*}
   \mathcal{G}(N,M,K,L)=\sum_{\substack{\sqrt{n}+\sqrt{m}\not=\sqrt{k}+\sqrt{\ell}\\
                              n\sim N,m\sim M,k\sim K,\ell\sim L\\1\leqslant N\leqslant M\leqslant y\\1\leqslant K\leqslant L\leqslant y}}
    g\cdot \min\bigg(T^{2\kappa},\frac{T^{2\kappa-1/2}}{|\sqrt{n}+\sqrt{m}-\sqrt{k}-\sqrt{\ell}|}\bigg).
\end{equation*}
If $M\geqslant200L$, then $|\sqrt{n}+\sqrt{m}-\sqrt{k}-\sqrt{\ell}|\gg M^{1/2}$, so the trivial estimate yields
\begin{equation*}
  \mathcal{G}(N,M,K,L)\ll \frac{T^{2\kappa-1/2+\varepsilon}NMKL}{(NMKL)^{3/4}M^{1/2}}\ll
      T^{2\kappa-1/2+\varepsilon}y^{1/2}\ll T^{2\kappa-1/8+\varepsilon}.
\end{equation*}
If $L\geqslant200M$, we can get the same estimate.  So later we always suppose that $M\asymp L$.
Let $\eta=\sqrt{n}+\sqrt{m}-\sqrt{k}-\sqrt{\ell}$. Write
\begin{equation}\label{G-fenjie}
  \mathcal{G}(N,M,K,L)=\mathcal{G}_1+\mathcal{G}_2+\mathcal{G}_3,
\end{equation}
where
\begin{eqnarray*}
   \mathcal{G}_1 & := & T^{2\kappa}\sum_{0<|\eta|\leqslant T^{-1/2}}g, \\
   \mathcal{G}_2 & := & T^{2\kappa-1/2}\sum_{T^{-1/2}<|\eta|\leqslant 1}g|\eta|^{-1}, \\
   \mathcal{G}_3 & := & T^{2\kappa-1/2}\sum_{|\eta|>1}g|\eta|^{-1}.
\end{eqnarray*}
We estimate $\mathcal{G}_1$ first. By Lemma \ref{Zhai-lemma-5}, we get
\begin{eqnarray}\label{G_1-111}
   \mathcal{G}_1 & \ll & \frac{T^{2\kappa+\varepsilon}}{(NMKL)^{3/4}}\mathscr{A}_1\big(N,M,K,L;T^{-1/2}\Big) \nonumber  \\
     & \ll & \frac{T^{2\kappa+\varepsilon}}{(NMKL)^{3/4}}\big(T^{-1/2}L^{1/2}NMK+NKL^{1/2}\big)    \nonumber  \\
     & \ll & T^{2\kappa-1/2+\varepsilon}(NK)^{1/4}+T^{2\kappa+\varepsilon}(NK)^{1/4}L^{-1}     \nonumber  \\
     & \ll & T^{2\kappa-1/2+\varepsilon}y^{1/2}+T^{2\kappa+\varepsilon}(NK)^{1/4}L^{-1}        \nonumber  \\
     & \ll & T^{2\kappa-1/8+\varepsilon}+T^{2\kappa+\varepsilon}(NK)^{1/4}L^{-1}.
\end{eqnarray}
On the other hand, by Lemma \ref{Zhai-lemma-3}, without loss of generality, we assume that $N\leqslant K\leqslant L$ and obtain
\begin{eqnarray}\label{G-1-222}
  \mathcal{G}_1 & \ll & \frac{T^{2\kappa+\varepsilon}}{(NMKL)^{3/4}}\mathscr{A}_{-}(N,M,K,L;T^{-1/2})   \nonumber  \\
  & \ll & \frac{T^{2\kappa+\varepsilon}}{(NMKL)^{3/4}} \big(T^{-1/8}N^{7/8}+N^{1/2}\big)
          \big(T^{-1/8}K^{7/8}+K^{1/2}\big)\big(T^{-1/4}L^{7/4}+L\big)      \nonumber  \\
  & \ll & T^{2\kappa+\varepsilon}(NK)^{-1/4}L^{-1/2}\big(T^{-1/8}N^{3/8}+1\big)\big(T^{-1/8}K^{3/8}+1\big)\big(T^{-1/4}L^{3/4}+1\big) \nonumber  \\
  & \ll & T^{2\kappa+\varepsilon}(NK)^{-1/4}L^{-1/2}
          \big(T^{-1/4}(NK)^{3/8}+T^{-1/8}K^{3/8}+1\big) \big(T^{-1/4}L^{3/4}+1\big)    \nonumber  \\
  & \ll & T^{2\kappa-1/4+\varepsilon}(NK)^{1/8}L^{-1/2}    \nonumber  \\
  &     & +T^{2\kappa+\varepsilon}(NK)^{-1/4}L^{-1/2} \big(T^{-1/8}K^{3/8}+1\big) \big(T^{-1/4}L^{3/4}+1\big)  \nonumber  \\
  & \ll & T^{2\kappa-1/4+\varepsilon}L^{-1/4}+T^{2\kappa+\varepsilon}(NK)^{-1/4}L^{-1/2}\big(T^{-3/8}L^{9/8}+1\big)    \nonumber  \\
  & \ll & T^{2\kappa-1/4+\varepsilon}+T^{2\kappa+\varepsilon}(NK)^{-1/4}L^{-1/2}\big(T^{-3/8}L^{9/8}+1\big).
\end{eqnarray}
From (\ref{G_1-111}) and (\ref{G-1-222}), we get
\begin{equation*}
   \mathcal{G}_1  \ll  T^{2\kappa-1/8+\varepsilon}+T^{2\kappa+\varepsilon}\cdot
   \min\bigg(\frac{(NK)^{1/4}}{L},\frac{T^{-3/8}L^{9/8}+1}{(NK)^{1/4}L^{1/2}}\bigg).
\end{equation*}
\textbf{Case 1} If $L\gg T^{1/3}$, then $T^{-3/8}L^{9/8}\gg1,$ we get
\begin{eqnarray}\label{G_1-case-1}
  \mathcal{G}_1 & \ll & T^{2\kappa-1/8+\varepsilon}+T^{2\kappa+\varepsilon}\cdot
   \min\bigg(\frac{(NK)^{1/4}}{L},\frac{T^{-3/8}L^{9/8}}{(NK)^{1/4}L^{1/2}}\bigg) \nonumber \\
                         & \ll & T^{2\kappa-1/8+\varepsilon}+T^{2\kappa+\varepsilon}\bigg(\frac{(NK)^{1/4}}{L}\bigg)^{1/2}
          \bigg(\frac{T^{-3/8}L^{9/8}}{(NK)^{1/4}L^{1/2}}\bigg)^{1/2}   \nonumber \\
    & \ll & T^{2\kappa-1/8+\varepsilon}+T^{2\kappa-3/16+\varepsilon}L^{-3/16}\ll T^{2\kappa-1/8+\varepsilon}.
\end{eqnarray}
\textbf{Case 2} If $L\ll T^{1/3}$, then $T^{-3/8}L^{9/8}\ll1.$
By noting that $M\asymp L\asymp\max(N,M,K,L)$ and Lemma \ref{kong-lamma}, we have
\begin{equation*}
 T^{-1/2}\gg|\eta|\gg(nmk\ell)^{-1/2}\max(n,m,k,\ell)^{-3/2}\asymp (NK)^{-1/2}L^{-5/2}.
\end{equation*}
Hence, we obtain
\begin{eqnarray}\label{G_1-case-2}
   \mathcal{G}_1 & \ll & T^{2\kappa-1/8+\varepsilon}+T^{2\kappa+\varepsilon}\min\bigg(\frac{(NK)^{1/4}}{L},\frac{1}{(NK)^{1/4}L^{1/2}}\bigg)
                           \nonumber \\
   & \ll & T^{2\kappa-1/8+\varepsilon}+T^{2\kappa+\varepsilon}\bigg(\frac{(NK)^{1/4}}{L}\bigg)^{1/4}\bigg(\frac{1}{(NK)^{1/4}L^{1/2}}\bigg)^{3/4}
                            \nonumber \\
   & = & T^{2\kappa-1/8+\varepsilon}+T^{2\kappa+\varepsilon}(NK)^{-1/8}L^{-5/8}
                            \nonumber \\
   & \ll & T^{2\kappa-1/8+\varepsilon}+T^{2\kappa+\varepsilon}(T^{-1/2})^{1/4}\ll T^{2\kappa-1/8+\varepsilon}.
\end{eqnarray}
Combining (\ref{G_1-case-1}) and (\ref{G_1-case-2}), we get
\begin{equation}\label{G-1-estimate}
   \mathcal{G}_1\ll T^{2\kappa-1/8+\varepsilon}.
\end{equation}

Now, we estimate $\mathcal{G}_2$. By a splitting argument, we get that there exists some $\delta$ satisfying $T^{-1/2}\ll \delta\ll 1$ such that
\begin{equation*}
  \mathcal{G}_2\ll\frac{T^{2\kappa-1/2+\varepsilon}}{(NMKL)^{3/4}\delta}\times\sum_{\substack{\delta<|\eta|\leqslant2\delta\\ \eta\not=0}}1.
\end{equation*}
By Lemma\ref{Zhai-lemma-5}, we get
\begin{eqnarray}\label{G-2-111}
  \mathcal{G}_2 & \ll & \frac{T^{2\kappa-1/2+\varepsilon}}{(NMKL)^{3/4}\delta}\mathscr{A}_1(N,M,K,L;2\delta)
                               \nonumber \\
     & \ll & \frac{T^{2\kappa-1/2+\varepsilon}}{(NMKL)^{3/4}\delta} (\delta L^{1/2}NMK+NKL^{1/2})
                                \nonumber \\
     & = &   T^{2\kappa-1/2+\varepsilon}(NK)^{1/4}+T^{2\kappa-1/2+\varepsilon}\delta^{-1}(NK)^{1/4}L^{-1}
                                \nonumber \\
     & \ll &  T^{2\kappa-1/2+\varepsilon}y^{1/2}+T^{2\kappa-1/2+\varepsilon}\delta^{-1}(NK)^{1/4}L^{-1}
                                \nonumber \\
     & \ll &  T^{2\kappa-1/8+\varepsilon}+T^{2\kappa-1/2+\varepsilon}\delta^{-1}(NK)^{1/4}L^{-1} .
\end{eqnarray}
On the other hand, by Lemma \ref{Zhai-lemma-3}, without loss of generality, we assume that $N\leqslant K\leqslant L$ and obtain
\begin{eqnarray}\label{G-2-222}
  \mathcal{G}_2 & \ll & \frac{T^{2\kappa-1/2+\varepsilon}}{(NMKL)^{3/4}\delta}\times\mathscr{A}_{-}(N,M,K,L;2\delta)
                           \nonumber   \\
   & \ll & \frac{T^{2\kappa-1/2+\varepsilon}}{(NMKL)^{3/4}\delta} \big(\delta^{1/4}N^{7/8}+N^{1/2}\big)
           \big(\delta^{1/4}K^{7/8}+K^{1/2}\big)\big(\delta^{1/2}L^{7/4}+L\big)
                            \nonumber   \\
   & \ll & T^{2\kappa-1/2+\varepsilon}(NK)^{-1/4}L^{-1/2}\delta^{-1}\big(\delta^{1/4}N^{3/8}+1\big)
           \big(\delta^{1/4}K^{3/8}+1\big)\big(\delta^{1/2}L^{3/4}+1\big)
                           \nonumber   \\
   & \ll &  T^{2\kappa-1/2+\varepsilon}(NK)^{-1/4}L^{-1/2}\delta^{-1}
            \big(\delta^{1/2}(NK)^{3/8}+\delta^{1/4}K^{3/8}+1\big) \big(\delta^{1/2}L^{3/4}+1\big)
                           \nonumber   \\
   & \ll & T^{2\kappa-1/2+\varepsilon}(NK)^{1/8}L^{-1/2}\delta^{-1/2}
                           \nonumber   \\
   &     &  +T^{2\kappa-1/2+\varepsilon}(NK)^{-1/4}L^{-1/2}\delta^{-1}\big(\delta^{1/4}K^{3/8}+1\big) \big(\delta^{1/2}L^{3/4}+1\big)
                            \nonumber   \\
   & \ll & T^{2\kappa-1/4+\varepsilon}L^{-1/4}+T^{2\kappa-1/2+\varepsilon}(NK)^{-1/4}L^{-1/2}\delta^{-1}\big(\delta^{3/4}L^{9/8}+1\big).
\end{eqnarray}
From (\ref{G-2-111}) and (\ref{G-2-222}), we get
\begin{equation*}
 \mathcal{G}_2  \ll  T^{2\kappa-1/8+\varepsilon}+T^{2\kappa-1/2+\varepsilon}\delta^{-1}
                         \cdot\min\bigg(\frac{(NK)^{1/4}}{L},\frac{\delta^{3/4}L^{9/8}+1}{(NK)^{1/4}L^{1/2}}\bigg).
\end{equation*}
\textbf{Case 1} If $\delta\gg L^{-3/2}$, then $\delta^{3/4}L^{9/8}\gg1$, we get (recall $\delta\gg T^{-1/2}$)
\begin{eqnarray}\label{G_2-case-1}
   \mathcal{G}_2 & \ll & T^{2\kappa-1/8+\varepsilon}+T^{2\kappa-1/2+\varepsilon}\delta^{-1}
                         \cdot\min\bigg(\frac{(NK)^{1/4}}{L},\frac{\delta^{3/4}L^{9/8}}{(NK)^{1/4}L^{1/2}}\bigg)
                          \nonumber   \\
              & \ll &  T^{2\kappa-1/8+\varepsilon}+ T^{2\kappa-1/2+\varepsilon}\delta^{-1}
                      \bigg(\frac{(NK)^{1/4}}{L}\bigg)^{1/2}   \bigg(\frac{\delta^{3/4}L^{9/8}}{(NK)^{1/4}L^{1/2}}\bigg)^{1/2}
                         \nonumber   \\
              & \ll & T^{2\kappa-1/8+\varepsilon} +T^{2\kappa-1/2+\varepsilon}\delta^{-5/8}L^{-3/16}
                         \nonumber   \\
              & \ll & T^{2\kappa-1/8+\varepsilon}+T^{2\kappa-1/2+\varepsilon}T^{5/16}L^{-3/16}\ll T^{2\kappa-1/8+\varepsilon}.
\end{eqnarray}
\textbf{Case 2} If $\delta\ll L^{-3/2}$, then $\delta^{3/4}L^{9/8}\ll1$. By Lemma \ref{kong-lamma}, we have
\begin{equation*}
   \delta\gg|\eta|\gg(nmk\ell)^{-1/2}\max(n,m,k,\ell)^{-3/2}\asymp (NK)^{-1/2}L^{-5/2}.
\end{equation*}
 Therefore, we obtain (recall $\delta\gg T^{-1/2}$)
\begin{eqnarray} \label{G_2-case-2}
  \mathcal{G}_2 & \ll & T^{2\kappa-1/8+\varepsilon}+T^{2\kappa-1/2+\varepsilon}\delta^{-1}\cdot
                         \min\bigg(\frac{(NK)^{1/4}}{L},\frac{1}{(NK)^{1/4}L^{1/2}}\bigg)
                              \nonumber   \\
    & \ll & T^{2\kappa-1/8+\varepsilon}+T^{2\kappa-1/2+\varepsilon}\delta^{-1}\bigg(\frac{(NK)^{1/4}}{L}\bigg)^{1/4}
             \bigg(\frac{1}{(NK)^{1/4}L^{1/2}}\bigg)^{3/4}
                              \nonumber   \\
    & \ll & T^{2\kappa-1/8+\varepsilon}+T^{2\kappa-1/2+\varepsilon}\delta^{-1}(NK)^{-1/8}L^{-5/8}
                             \nonumber   \\
     & \ll & T^{2\kappa-1/8+\varepsilon}+T^{2\kappa-1/2+\varepsilon}\delta^{-1}\delta^{1/4}
                               \nonumber   \\
     & \ll & T^{2\kappa-1/8+\varepsilon}+T^{2\kappa-1/2+\varepsilon}T^{3/8}\ll  T^{2\kappa-1/8+\varepsilon}.
\end{eqnarray}
Combining (\ref{G_2-case-1}) and (\ref{G_2-case-2}), we get
\begin{equation}\label{G-2-estimate}
   \mathcal{G}_2\ll T^{2\kappa-1/8+\varepsilon}.
\end{equation}

 For $\mathcal{G}_3$, by a splitting argument and Lemma \ref{Zhai-lemma-5} again, we get
\begin{eqnarray} \label{G-3-estimate}
   \mathcal{G}_3 & \ll & \frac{T^{2\kappa-1/2+\varepsilon}}{(NMKL)^{3/4}\delta}\times
                         \sum_{\substack{\delta<|\eta|\leqslant2\delta\\ \delta\gg 1}}1
                             \nonumber \\
    & \ll & \frac{T^{2\kappa-1/2+\varepsilon}}{(NMKL)^{3/4}\delta}\cdot \delta L^{1/2}NMK\ll  T^{2\kappa-1/2+\varepsilon}(NK)^{1/4}  \nonumber \\
    & \ll &  T^{2\kappa-1/2+\varepsilon}y^{1/2}\ll T^{2\kappa-1/8+\varepsilon}.
\end{eqnarray}
Comnining (\ref{S_2-estimate}), (\ref{G-fenjie}), (\ref{G-1-estimate}), (\ref{G-2-estimate}) and (\ref{G-3-estimate}), we get
\begin{equation}\label{S_2-upper}
   \int_T^{2T}S_2(x)\mathrm{d}x\ll T^{2\kappa-1/8+\varepsilon}.
\end{equation}
In the same way, we can prove that
\begin{equation}\label{S_3-upper}
   \int_T^{2T}S_3(x)\mathrm{d}x\ll T^{2\kappa-1/8+\varepsilon}.
\end{equation}

  From (\ref{A^4-asymp})-(\ref{S_4-estimate}), (\ref{S_2-upper}) and (\ref{S_3-upper}), we get
\begin{equation}
   \int_T^{2T}A^4(x)\mathrm{d}x=\frac{3}{32\pi^4}s_{4;2}(\tilde{a})\int_{T}^{2T}x^{2\kappa-1}\mathrm{d}x+O(T^{2\kappa-1/8+\varepsilon}),
\end{equation}
which implies Theorem \ref{4-power-theorem} immediately.

\bigskip
\bigskip

\textbf{Acknowledgement}

   The authors would like to express the most and the greatest sincere gratitude to Professor Wenguang Zhai for his valuable
advice and constant encouragement.

\end{document}